\documentclass[a4paper,draft]{amsproc}
\usepackage{amssymb}
\usepackage[hyphens]{url} 

\theoremstyle{plain}
 \newtheorem{theorem}{Theorem}[section]
 \newtheorem{proposition}{Proposition}[section]
 \newtheorem{lemma}{Lemma}[section]
 \newtheorem{corollary}{Corollary}[section]
\theoremstyle{definition}

\theoremstyle{remark}

 \numberwithin{equation}{section}

\renewcommand{\leq}{\leqslant}
\renewcommand{\geq}{\geqslant}

\title[ On a trigonometric sum/ M. Goubi]{On a finite trigonometric sum related to Dedekind sum}



\author{\bfseries Mouloud  Goubi} 
\address{Mouloud Goubi\\
Department of Mathematics \\
University of UMMTO RP. 15000\\
Tizi-ouzou, Algeria\\
Laboratoire d'Alg\`ebre et Th\'eorie des Nombres, USTHB Alger}
\email{mouloud.goubi@ummto.dz}

\begin{document}

\vspace{18mm} \setcounter{page}{1} \thispagestyle{empty}

\begin{abstract}
Finite trigonometric sums appear in various branches of Physics,
Mathematics and their applications. For $p,q$ to coprime positive
integers and $r\in\mathbb{Z}$ we consider the finite trigonometric
sums involving the product of three trigonometric functions
$$S\left(p,q,r\right)=\sum_{k=1}^{p-1}\cos\left(\frac{\pi
kr}{p}\right)\csc\left(\frac{\pi kq}{p}\right)\csc\left(\frac{\pi
k}{p}\right).\:$$ In this work we establish a reciprocity law
satisfied by $S\left(p,q,r\right)$ in some particular cases. And we
compute explicitly the value of the sum
$$S\left(p,1,0\right)=\sum_{k=1}^{p-1}\csc^2\left(\frac{\pi
k}{p}\right).\:$$
\end{abstract}

\maketitle

\section{Introduction and main results}
Let as remember the cosecant function $\csc$ defined by
$$\csc\theta=\frac{1}{\sin\theta}.\:$$
From which the cosecant numbers are extracted, as the rational
coefficients of it's power series expansion, a generalization of
these numbers is given in the paper \cite{Kow} of V.~Kowalenko.\\

For $p,q$ two coprime positive integers and $r\in\mathbb{Z}$ we
consider the following finite trigonometric sum.
\begin{align}\label{equa1}
S\left(p,q,r\right)=\sum_{k=1}^{p-1}\cos\left(\frac{\pi
kr}{p}\right)\csc\left(\frac{\pi kq}{p}\right)\csc\left(\frac{\pi
k}{p}\right)
\end{align}
Let $0\leq\overline{r}\leq 2p-1$ the represent of the class of $r$
in $\mathbb{Z}/2p\mathbb{Z}$, then
$$S\left(p,q,r\right)=S\left(p,q,\overline{r}\right)$$.\\
For example if $r\equiv0\pmod{2p}$;
$$S\left(p,q,r\right)=\sum_{k=1}^{p-1}\csc\left(\frac{\pi kq}{p}\right)\csc\left(\frac{\pi
k}{p}\right)\:$$\\

The periodicity of cosine stays that there are exactly $2p$ finite
trigonometric sums of this kind for a fixed positive integers
$p,q$.\\
If $q=\overline{q}\pmod{2p}$ then
$$S\left(p,q,r\right)=S\left(p,\overline{q},r\right)\:$$
and for a fixed $p$ and $q,r$ are arbitrary integers such that $q$
is not multiple of $p$ their exist exactly $2p\left(2p-1\right)$
finite trigonometric sums $S\left(p,q,r\right)$.\\

Now if $r\equiv0\pmod p$ and $q\equiv1\pmod{2p},$ we get
$$S\left(p,q,r\right)=\sum_{k=1}^{p-1}\csc^2\left(\frac{\pi k}{p}\right),\:$$
and if $q=1$ and $r\equiv1\pmod{p}$ then
$$S\left(p,q,0\right)=\sum_{k=1}^{p-1}\left(-1\right)^k\csc^2\left(\frac{\pi k}{p}\right).\:$$
Finally for $r=q$ the considered trigonometric sum will be written
in the following form
$$\sum_{k=1}^{p-1}\cot\left(\frac{\pi kq}{p}\right)\csc\left(\frac{\pi
k}{p}\right).\:$$\\

The problem of evaluation of the sum
$$\sum_{k=1}^{p-1}\csc^{2m}\left(\frac{\pi k}{p}\right)\:$$
steal open. A partial answer is given in the work of the physician
N. Gauthier and Paul S. Bruckman \cite{Gaut} at order $Q$, where
$Q=\frac{p-1}{2}$ if $p$ odd
and $Q=\frac{p-2}{2}$ otherwise.\\
$$\sum_{k=1}^{Q}\csc^{2m}\left(\frac{\pi
k}{p}\right)=\sum_{k=1}^{m}\left(2k-1\right)!\varphi_{k-1,m-1}J_{2k}\left(p\right),
~m\geq1$$ where
$$\varphi_{r,m}=\frac{s_{r,m}}{\left(2m+1\right)!};~0\leq r\leq m; m\geq1,\:$$
$s_{r,m}$ is the sum of all the possible distinct products of the
following numbers $4.1^2, 4.2^2,...,4.2^m$ and
$$J_{2k}\left(p\right)=\left(\frac{p}{\pi}\right)^{2k}\sum_{r=1}^{Q}\sum_{n}\left(r-np\right)^{-2k}.$$

But the value of the sum of maximal order of even integral powers of
the secant $\sum_{k=1}^{p-1}\sec^{2m}\left(\frac{\pi k}{p}\right)\:$
is expressed by \cite[p. 1]{Fons17}
$$\sum_{k=1}^{p-1}\sec^{2m}\left(\frac{\pi
k}{p}\right)=p\sum_{k=1}^{2m-1}\left(-1\right)^{m+k}\left(m-1+kp\atop
2m-1\right)\sum_{j=k}^{2m-1}\left(2m\atop j+1\right)\:$$ The problem
of finite sums with negative powers of cosecant or secant is
completely resolved in \cite{Fons13} and \cite[Theorem2.1
p.4]{Fons17}, we copy here the result for cosecant
\begin{align*}
\sum_{k=0}^{p-1}\csc^{-2m}\left(\frac{\pi k}{p}\right)= \left\{
\begin{array}{ccc}
2^{1-2m}p\left(\left(2m-1\atop m-1\right)+\sum_{n=1}^{[m/p]}\left(-1\right)^{pn}\left(2m\atop m-pn\right)\right)\ ,&\quad \textrm{if}\ m\geq p,\\
2^{1-2m}p\left(\left(2m-1\atop m-1\right)\right)\ , &\quad\ m<p.
\end{array}
\right.
\end{align*}

Let $m,n,l\in\mathbb{Z}$, the last finite trigonometric sums can be
extended to the general case
\begin{equation}\label{equa18}
S_{n,m,l}\left(p,q,r\right)=\sum_{k=1}^{p-1}\cos^n\left(\frac{\pi
kr}{p}\right)\csc^m\left(\frac{\pi
kq}{p}\right)\csc^l\left(\frac{\pi k}{p}\right),
\end{equation}
and then
$$S\left(p,q,r\right)=S_{1,1,1}\left(p,q,r\right).$$
Since $p-k$ runs all the numbers from $1$ to $p-1$ for
$k\in\left\{1,2,3,...,p-1\right\}$ then the sum
$S_{n,m,l}\left(p,q,r\right)$ can be written in the following form

$$S_{n,m,l}\left(p,q,r\right)=\sum_{k=1}^{p-1}\cos^n\left(\pi r-\frac{\pi
kr}{p}\right)\csc^m\left(\pi q-\frac{\pi
kq}{p}\right)\csc^l\left(\pi -\frac{\pi k}{p}\right).\:$$ but
$\cos\left(\pi k-\theta\right)=\left(-1\right)^k\cos\theta$ and
$\sin\left(\pi k-\theta\right)=\left(-1\right)^{k+1}\sin\theta$ then

$$S_{n,m,l}\left(p,q,r\right)=\sum_{k=1}^{p-1}\left(-1\right)^{rn+\left(q+1\right)m}\cos^n\left(\frac{\pi
kr}{p}\right)\csc^m\left(\frac{\pi
kq}{p}\right)\csc^l\left(\frac{\pi k}{p}\right),\:$$ which means
that
$$\left(1-\left(-1\right)^{rn+\left(q+1\right)m}\right)S_{n,m,l}\left(p,q,r\right)=0\:$$
and for $rn+\left(q+1\right)m$ odd number we get
$$S_{n,m,l}\left(p,q,r\right)=0\:$$

On general, evaluating the reciprocity law of
$S_{n,m,l}\left(p,q,r\right)$ steal an open
problem, for which we need more tools to resolve.\\

For more background about finite trigonometric sums in literature we
refer to \cite{Bern}, \cite{Chan} and \cite{Wil}. The methods of
computation based only on the residue theorem from complex analysis.
The well-known classical Dedekind sum is
\begin{equation}\label{equa2}
s\left(q,p\right)=\sum_{k=1}^{p-1}\left(\left(\frac{kq}{p}\right)\right)\left(\left(\frac{k}{p}\right)\right)
\end{equation}
where $q,p$ are coprime positive integers and
\begin{align*}
\left(\left(x\right)\right):= \left\{
\begin{array}{ccc}
x-[x]-\frac{1}{2}\ ,&\quad \textrm{if}\ x\ \textrm{is not an integer}, \\
0\ , &\quad  \textrm{ otherwise}.
\end{array}
\right.
\end{align*}
In terms of fractional part function $\left\{.\right\}$ $(ie.
\left\{x\right\}=x-[x] )$ we can write
\begin{equation}\label{equa3}
s\left(q,p\right)=\frac{1-p}{4}+\sum_{k=1}^{p-1}\left\{\frac{kq}{p}\right\}\left\{\frac{k}{p}\right\}
\end{equation}
From the property that $\left\{x\right\}+\left\{-x\right\}=1$, If
$0<q<p$, we get
\begin{equation} \label{equa16}
s\left(p-q,p\right)=\frac{p-1}{2}-s\left(q,p\right).
\end{equation}

$s\left(q,p\right)$ satisfies the following reciprocity law
\cite[p.4]{Rademacher72}
\begin{equation}\label{equa17}
 s\left(q,p\right)+
 s\left(p,q\right)=-\frac{1}{4}+\frac{1}{12}\left(\frac{q}{p}+\frac{p}{q}+\frac{1}{pq}\right)
\end{equation}

The reciprocity law inducts an expression for $s\left(p-q,p\right)$
on function of $s\left(q,p\right)$.
\begin{align}\label{equa4}
 s\left(p,p-q\right)=s\left(q,p\right)+\frac{1}{4}+\frac{1}{12p\left(p-q\right)}\left(-6p^3+6p^2q+2p^2+q^2-2pq+1\right)
\end{align}

This sum is interesting, since it can only be traduced to a finite
trigonometric sum of product of cotangent functions
\cite[p.18]{Rademacher72}
\begin{align}\label{equa5}
s\left(q,p\right)=\frac{1}{4p}\sum_{k=1}^{p-1}\cot\left(\frac{\pi
kq}{p}\right)\cot\left(\frac{\pi k}{p}\right)
\end{align}

The objective of this paper is to establish a reciprocity law for
the trigonometric sums $S\left(p,q,q\pm1\right)$ and
$S\left(p,p-q,p-q\pm1\right)$ , and get explicit evaluation of the
sum $$S\left(p,1,0\right)=\sum_{k=1}^{p-1}\csc^2\left(\frac{\pi
k}{p}\right).\:$$

\begin{theorem}\label{th1}
For $p,q$ two coprime positive integers the following statement is
true
\begin{align}\label{equa6}\\
\nonumber
q[S\left(p,q,q+1\right)+S\left(p,q,q-1\right)]+p[S\left(q,p,p+1\right)+S\left(q,p,p-1\right)]&=
-2pq+\frac{2}{3}\left(p^2+q^2+1\right)
\end{align}
And if $q<p$ we get
\begin{align}\label{equa7}
p[S\left(q,p,p+1\right)+S\left(q,p,p-1\right)]-q[S\left(p,p-q,p-q+1\right)+S\left(p,p-q,p-q-1\right)]&\\
\nonumber=2pq-4p^2q+\frac{2}{3}\left(p^2+q^2+1\right)
\end{align}
\end{theorem}
\begin{corollary}\label{coro1}
Let $\left(p,q\right)=1$ if $q<p$ then
\begin{align}\label{equa8}
S\left(p,q,q+1\right)+S\left(p,q,q-1\right)+S\left(p,p-q,p-q+1\right)+S\left(p,p-q,p-q-1\right)&=
4pq\left(p-1\right)
\end{align}
\end{corollary}
The following theorem gives another reciprocity law and compute the
sum of the maximal order of the square of cosecant.
\begin{theorem}\label{th2}
Let $q\equiv1\pmod p$ then
\begin{align}\label{equa9}
S\left(q,p,p+1\right)+S\left(q,p,p-1\right)=\frac{2}{3}\left(q+\frac{p^2}{q}+\frac{1}{q}-p^2-2\right)
\end{align}
and
\begin{align}\label{equa10}
\sum_{k=1}^{p-1}\csc^2\left(\frac{\pi k}{p}\right)=\frac{p^2-1}{3}
\end{align}
\end{theorem}

\section{proof of main results}
\begin{lemma}\label{lem1}
\begin{equation} \label{equa11} \cot
q\beta\cot\beta=\frac{\cos\left(q+1\right)\beta+\cos\left(q-1\right)\beta}{2\sin
q\beta\sin\beta}
\end{equation}
\end{lemma}
\begin{proof}
From the well-known trigonometric formula
\begin{align*}
\cos\alpha\cos\beta=\frac{1}{2}\left(\cos\left(\alpha+\beta\right)+\cos\left(\alpha-\beta\right)\right)
\end{align*}
we get
\begin{align*}
\cot\alpha\cot\beta=\frac{\cos\alpha\cos\beta}{\sin\alpha\sin\beta}=\frac{\cos\left(\alpha+\beta\right)+\cos\left(\alpha-\beta\right)}{2\sin\alpha\sin\beta}
\end{align*}
and the result is deduced by taking $\alpha=q\beta$.
\end{proof}
\begin{lemma}\label{lem2}
\begin{align}\label{equa12}
\sum_{k=1}^{p-1}\cot\left(\frac{\pi kq}{p}\right)\cot\left(\frac{\pi
k}{p}\right)&=\frac{1}{2}\left\{S\left(p,q,q+1\right)+S\left(p,q,q-1\right)\right\}
\end{align}
\end{lemma}
\begin{proof}
\begin{align*}
\sum_{k=1}^{p-1}\cot\left(\frac{\pi kq}{p}\right)\cot\left(\frac{\pi
k}{p}\right)&=\frac{1}{2}\sum_{k=1}^{p-1}\left(\cos\left(\frac{\pi
k\left(q+1\right)}{p}\right)+\cos\left(\frac{\pi
k\left(q-1\right)}{p}\right)\right)\csc\left(\frac{\pi
kq}{p}\right)\csc\left(\frac{\pi k}{p}\right)
\end{align*}
The decomposition of the right hand of the equality on two sums
gives the result.
\end{proof}
Using the relation \eqref{equa12} we deduce that

\begin{equation}\label{equa12bis}
s\left(q,p\right)=\frac{1}{8p}\left\{S\left(p,q,q+1\right)+S\left(p,q,q-1\right)\right\}
\end{equation}
\begin{lemma}\label{lem3}
Let $q\equiv1\pmod p$ then
\begin{align}\label{equa13}
s\left(q,p\right)=\frac{\left(p-1\right)\left(p-2\right)}{12p}
\end{align}
and
\begin{align}\label{equa14}
s\left(p,q\right)=\frac{1}{12}\left(\frac{q-2}{p}+\frac{p}{q}+\frac{1}{pq}-p\right)
\end{align}
\end{lemma}
\begin{proof}
For $q\equiv 1\pmod p$, $s\left(q,p\right)=s\left(1,p\right)$ is
explicitly evaluated and
$$s\left(q,p\right)=\sum_{k=1}^{p-1}\left(\frac{k}{p}-\frac{1}{2}\right)^2\:$$
which becomes
$$s\left(q,p\right)=\frac{p-1}{4}-\frac{1}{p}\sum_{k=1}^{p-1}k+\frac{1}{p^2}\sum_{k=1}^{p-1}k^2\:$$
It's well-known that
$$\sum_{k=1}^{p-1}k=\frac{\left(p-1\right)p}{2}\:$$ and
$$\sum_{k=1}^{p-1}k^2=\frac{p\left(p-1\right)\left(2p-1\right)}{6}.\:$$
and the result \eqref{equa13} follows.\\
The second result \eqref{equa14} is the consequence of the relation
\eqref{equa13} and the reciprocity law \eqref{equa17}.
\end{proof}
\subsubsection{Proof of Theorem\ref{th1}} From the relation
\eqref{equa12} Lemma\ref{lem2} and the expression \eqref{equa5} we
get
$$s\left(q,p\right)=\frac{1}{8p}\left\{S\left(p,q,q+1\right)+S\left(p,q,q-1\right)\right\}\:$$
By symmetry we get the similar expression
$$s\left(p,q\right)=\frac{1}{8q}\left\{S\left(q,p,p+1\right)+S\left(q,p,p-1\right)\right\}.\:$$
The reciprocity formula \eqref{equa17} inducts
\begin{align*}
\frac{1}{8p}\left\{S\left(p,q,q+1\right)+S\left(p,q,q-1\right)\right\}+\frac{1}{8q}\left\{S\left(q,p,p+1\right)+S\left(q,p,p-1\right)\right\}&=\\
-\frac{1}{4}+\frac{1}{12}\left(\frac{q}{p}+\frac{p}{q}+\frac{1}{pq}\right).
\end{align*}
Multiplying this equality by $8pq$ we get the result \eqref{equa6}
Theorem\ref{th1}.\\

Only we have
$$s\left(p-q,p\right)=\frac{1}{8p}\left\{S\left(p,p-q,p-q+1\right)+S\left(p,p-q,p-q-1\right)\right\}\:$$
and from the relation \eqref{equa16} we get
$$S\left(q,p\right)=\frac{p-1}{2}-\frac{1}{8p}\left\{S\left(p,p-q,p-q+1\right)+S\left(p,p-q,p-q-1\right)\right\},\:$$
and then
$$\frac{1}{8p}\left\{S\left(p,q,q+1\right)+S\left(p,q,q-1\right)\right\}=
\frac{p-1}{2}-\frac{1}{8p}\left\{S\left(p,p-q,p-q+1\right)+S\left(p,p-q,p-q-1\right)\right\},\:$$
thus
$$q\left\{S\left(p,q,q+1\right)+S\left(p,q,q-1\right)\right\}=
4pq\left(p-1\right)-q\left\{S\left(p,p-q,p-q+1\right)+S\left(p,p-q,p-q-1\right)\right\},\:$$
and then
\begin{align*}
p\left\{S\left(q,p,p+1\right)+S\left(q,p,p-1\right)\right\}-q\left\{S\left(p,p-q,p-q+1\right)+S\left(p,p-q,p-q-1\right)\right\}&\\
=2pq-4p^2q+\frac{2}{3}\left(p^2+q^2+1\right)
\end{align*}
\subsubsection{Proof of Corollary\ref{coro1}}
The reciprocity theorem \eqref{equa8} of Corollary\ref{coro1} is
result of the relation \eqref{equa6} minus the relation
\eqref{equa7}
\subsubsection{Proof of Theorem\ref{th2}} From the
reciprocity theorem \label{equa3}
$$s\left(p,q\right)=-\frac{1}{4}+\frac{1}{12}\left(\frac{q}{p}+\frac{p}{q}+\frac{1}{pq}\right)-s\left(q,p\right)\:$$
using the relation \eqref{equa13} we deduce that
$$s\left(p,q\right)=-\frac{1}{4}+\frac{1}{12}\left(\frac{q}{p}+\frac{p}{q}+\frac{1}{pq}\right)-\frac{\left(p-1\right)\left(p-2\right)}{12p}\:$$
and
$$s\left(p,q\right)=-\frac{1}{4}+\frac{1}{12}\left(\frac{q}{p}+\frac{p}{q}+\frac{1}{pq}-p+3-\frac{2}{p}\right).\:$$
Thus
$$s\left(p,q\right)=\frac{1}{12}\left(\frac{q-2}{p}+\frac{p}{q}+\frac{1}{pq}-p\right).\:$$
From the relation \eqref{equa12bis} we get
$$\frac{1}{8p}\left\{S\left(q,p,p+1\right)+S\left(q,p,p-1\right)\right\}=\frac{1}{12}\left(\frac{q-2}{p}+\frac{p}{q}+\frac{1}{pq}-p\right)\:$$
and the result \eqref{equa9} follows.\\

Taking $q=1$, and combining the relation \eqref{equa12bis} and
equality \eqref{equa13} Lemma\ref{lem3} we get
$$S\left(p,1,2\right)+S\left(p,1,0\right)=\frac{2}{3}\left(p-1\right)\left(p-2\right),\:$$
then
$$\sum_{k=1}^{p-1}\frac{\cos\left(\frac{2\pi
k}{p}\right)+1}{\sin^2\left(\frac{\pi
k}{p}\right)}=\frac{2p^2-6p+4}{3}.\:$$ But

$$\cos\left(\frac{2\pi k}{p}\right)=1-2\sin^2\left(\frac{\pi k}{p}\right)\:$$
then
$$\sum_{k=1}^{p-1}\frac{2-2\sin^2\left(\frac{\pi
k}{p}\right)}{\sin^2\left(\frac{\pi
k}{p}\right)}=\frac{2p^2-6p+4}{3}.\:$$ and
$$2\sum_{k=1}^{p-1}\frac{1}{\sin^2\left(\frac{\pi
k}{p}\right)}=\frac{2p^2-6p+4}{3}+2\left(p-1\right).\:$$ Thus

 $$
\sum_{k=1}^{p-1}\frac{2}{\sin^2\left(\frac{\pi
k}{p}\right)}=\frac{p^2-3p+2}{3}+p-1\: $$ Finally
$$\sum_{k=1}^{p-1}\frac{1}{\sin^2\left(\frac{\pi
k}{p}\right)}=\frac{p^2-1}{3}.\:$$
\section{Additional result}
They are several different sums extracted from the classical
Dedekind sum which satisfy a reciprocity law for a particular case.
An example of a generalized cosecant sum \eqref{equa18} is given in
the following proposition
\begin{proposition}\label{prop1}
\begin{equation}\label{equa15}
2S_{3,1,1}\left(p,2,1\right)-S_{1,1,1}\left(p,2,1\right)=\frac{1}{6}p^2-p+\frac{5}{2}
\end{equation}
\end{proposition}
\begin{proof}
It's trivial that $s\left(p,2\right)=0$ and from the reciprocity law
\eqref{equa17} we get
$$s\left(2,p\right)=-\frac{1}{4}+\frac{1}{12}\left(\frac{p}{2}+\frac{5}{2p}\right).\:$$
Form the relation \eqref{equa12bis} we get
$$s\left(2,p\right)=\frac{1}{8p}\left(S\left(p,2,3\right)+S\left(p,2,1\right)\right)\:$$
and
\begin{align*}
S\left(p,2,3\right)+S\left(p,2,1\right)=\sum_{k=1}^{p-1}\left(\cos\left(\frac{3\pi
k}{p}\right)+\cos\left(\frac{\pi
k}{p}\right)\right)\csc\left(\frac{2\pi
k}{p}\right)\csc\left(\frac{\pi k}{p}\right).
\end{align*}
Using the trigonometric relation
$$\cos3\theta=4\cos^3\theta-3\cos\theta\:$$
\begin{align*}
S\left(p,2,3\right)+S\left(p,2,1\right)&=&4\sum_{k=1}^{p-1}\cos^3\left(\frac{\pi
k}{p}\right)\csc\left(\frac{2\pi k}{p}\right)\csc\left(\frac{\pi
k}{p}\right)\\
&-&2\sum_{k=1}^{p-1}\cos\left(\frac{\pi
k}{p}\right)\csc\left(\frac{2\pi k}{p}\right)\csc\left(\frac{\pi
k}{p}\right).
\end{align*}
We deduce that
$$4S_{3,1,1}\left(p,2,1\right)-2S_{1,1,1}\left(p,1,1\right)=8ps\left(2,p\right).\:$$
Replacing $s\left(2,p\right)$ by it's value we get the result
\eqref{equa15} Proposition\ref{prop1}.
 \end{proof}

Exploitation of $S\left(2,p\right)=0$ conduct to proof of the
following trigonometric formula
$$\cos\left(\frac{\pi\left(p+1\right)}{2}\right)+\cos\left(\frac{\pi\left(p-1\right)}{2}\right)=0\:$$
The argument is that
$$\frac{1}{16}\left\{S\left(2,p,p+1\right)+S\left(2,p,p-1\right)\right\}=s\left(2,p\right)=0\:$$ and then
$$\left\{\cos\left(\frac{\pi\left(p+1\right)}{2}\right)+\cos\left(\frac{\pi\left(p-1\right)}{2}\right)\right\}\csc\left(\frac{\pi}{2}\right)=0\:$$

\end{document}